\documentclass[11pt,a4paper]{amsart}
\usepackage{fullpage}              
              \usepackage{graphicx,verbatim,url,paralist}
\usepackage{amssymb}
\usepackage{epstopdf,xcolor,hyperref,graphicx,overpic}
\usepackage[normalem]{ulem}

\theoremstyle{plain}
\newtheorem{thm}{Theorem}[section]
\newtheorem*{thm*}{Theorem}
\newtheorem{thmi}{Theorem} 
\newtheorem{prop}[thm]{Proposition}
\newtheorem*{prop*}{Proposition}
\newtheorem{lemma}[thm]{Lemma}
\newtheorem*{lemma*}{Lemma}
\newtheorem*{claim*}{Claim}
\newtheorem{corollary}[thm]{Corollary}

\theoremstyle{definition}
\newtheorem{rmk}[thm]{Remark}

\newcommand{\R}{\mathbb{R}}

\newcommand{\tA}{\tilde{A}}
\newcommand{\tC}{\tilde{C}}
\newcommand{\tG}{\tilde{G}}
\newcommand{\App}{\mathcal{A}}
\newcommand{\Aff}{\mathbb{A}}
\newcommand{\br}{\mathbf{r}}

\newcommand{\CAT}{\operatorname{CAT}}

\newcommand{\Fix}{\operatorname{Fix}}

\newcommand{\op}{\mathrm{op}}

\title[Fixed points for group actions on $2$-dimensional affine buildings]{Fixed points for group actions on \\$2$-dimensional affine buildings}
\author{Jeroen Schillewaert, Koen Struyve and Anne Thomas}
\address{Jeroen Schillewaert, Department of Mathematics, University of Auckland, 38 Princes Street, 1010 Auckland, New Zealand 
\newline Koen Struyve, Burgemeester Couckestraat 4, 8720 Dentergem, Belgium
\newline Anne Thomas, School of Mathematics \& Statistics, Carslaw Building F07,  University of Sydney NSW 2006, Australia}
\email{j.schillewaert@auckland.ac.nz,  kstruy@gmail.com, anne.thomas@sydney.edu.au}
\thanks{This research of the first author is supported by the New Zealand Marsden fund through grant UOA-2122 and that of the third author by ARC grant DP180102437.}

\begin{document}
\parindent 0cm
\parskip 2mm

\begin{abstract}
We prove a local-to-global result for fixed points of finitely generated groups acting on 2-dimensional affine buildings of types $\tA_2$ and $\tC_2$.  Our proofs combine building-theoretic arguments with standard results for $\CAT(0)$ spaces.
\end{abstract}

\maketitle

\section{Introduction}

The study of local-to-global results for fixed points of groups acting on affine buildings originated with Serre, who proved such a result for simplicial trees~\cite[Corollary 3 of Section 6.5]{Serre} and introduced property (FA). Serre's result was extended by Morgan and Shalen to $\mathbb{R}$-trees~\cite[Proposition II.2.15]{MorganShalen}.  We prove a similar result for certain 2-dimensional affine buildings, possibly nondiscrete.   
We note that the results of Serre and Morgan--Shalen immediately imply the following theorem for buildings of type $\tilde{A}_1 \times \tilde{A}_1$, since such buildings are products of trees. 

\begin{thmi}\label{thm:main}  Let $G$ be a finitely generated group of automorphisms of a 2-dimensional affine building $X$ of type $\tA_2$ or $\tC_2$ (possibly nondiscrete). If every element of $G$ fixes a point of $X$, then $G$ fixes a point of $X$.
\end{thmi}

\begin{rmk}
If a $2$-dimensional affine building $X$ is the Bruhat--Tits building for a reductive group over a field with valuation, or $X$ is discrete, then $X$ has type $\tA_2$, $\tC_2$ or $\tG_2$.  We give a detailed explanation of why our techniques do not extend to type $\tG_2$ in Remarks~\ref{rem:G2local} and~\ref{rem:G2main}.
\end{rmk}

By applying Theorem~\ref{thm:main} to all finitely generated subgroups and then invoking a result of Caprace and Lytchak \cite[Theorem 1.1]{CL10}, we obtain the following corollary for non-finitely generated groups.  

\begin{corollary}\label{compactification}
Suppose a group $G$ acts on a complete 2-dimensional affine building $X$ of type $\tA_2$ or $\tC_2$ such that every element of $G$ fixes a point of~$X$. Then $G$ fixes a point in the bordification $\overline{X}=X\cup \partial X$ of $X$.
\end{corollary}

When $X$ is discrete, Corollary \ref{compactification} confirms Conjecture 1.2 of Marquis~\cite{Marquis} for this class of buildings. In this paper Marquis introduces Property (FB): every measurable action of a group~$G$ by type-preserving simplicial isometries on a finite rank discrete building stabilises a spherical residue. Since continuous actions are measurable, we see that for discrete groups this is a higher-rank analogue of Serre's property (FA).  As explained in~\cite[Remark 3.4]{Marquis}, Corollary~\ref{compactification} combined with the Morgan--Shalen result for trees implies a special case of Conjecture 1.1 of~\cite{Marquis}.
\begin{corollary}
Let $G$ be an almost connected locally compact group which acts measurably by type-preserving simplicial isometries on a discrete 2-dimensional affine building of type $\tA_2$ or~$\tC_2$.  Then $G$ fixes a point of $X$.
\end{corollary} 

Many other local-to-global results similar to Theorem~\ref{thm:main} are known. Parreau in \cite[Corollaire~3]{Parreau} proved a similar result for subgroups $\Gamma$ of connected reductive groups $\mathcal{G}$ over certain fields $F$, where $\Gamma$ is generated by a bounded subset of $\mathcal{G}(F)$ and the action is on the completion of the associated Bruhat--Tits building.  Breuillard and Fujiwara established a quantitative version of Parreau's result for discrete Bruhat--Tits buildings \cite[Theorem 7.16]{BreuillardFujiwara} and asked whether their result holds for the isometry group of an arbitrary affine building.  Leder and Varghese in \cite{LederVarghese}, using work of Sageev~\cite{Sageev}, obtained a similar result for groups acting on finite-dimensional $\CAT(0)$ cube complexes. However, a statement similar to Theorem \ref{thm:main} is false for infinite-dimensional $\CAT(0)$ cubical complexes, as shown by Osajda using actions of infinite free Burnside groups~\cite{Osajda}. 

While we were preparing this paper, Norin, Osajda and Przytycki~\cite{NOP} proved a result closely related to Theorem~\ref{thm:main}.

\begin{thm}[Theorem 1.1 of~\cite{NOP}] \label{NOP}
Let $X$ be a $\CAT(0)$ triangle complex and let $G$ be a finitely generated group acting on $X$ with no global fixed point. Assume that either
each element of~$G$ fixing a point of $X$ has finite order, or $X$ is locally finite, or $X$ has rational angles. Then $G$ has an element with no fixed point in $X$.
\end{thm}

\begin{rmk}
One of the most important classes of $\CAT(0)$ triangle complexes is that of $2$-dimensional discrete affine buildings, and since such buildings have rational angles, Theorem~\ref{NOP} applies to these.  In particular, it applies to discrete buildings of type $\tG_2$.  On the other hand,  Theorem~\ref{thm:main} is valid for both discrete and nondiscrete buildings (of types $\tA_2$ and $\tC_2$), and as explained below, some parts of our argument are straightforward in the discrete case. Our proof combines general $\CAT(0)$-space techniques with arguments specific to spherical and affine buildings, while \cite{NOP} uses Helly's theorem from~\cite{Ivanov} together with sophisticated results including Masur's theorem on periodic trajectories in rational billiards~\cite{Masur}, and Ballmann and Brin's methods for finding closed geodesics in $2$-dimensional locally $\CAT(0)$ complexes~\cite{BB95}. 
\end{rmk}

The next result follows from Theorem~\ref{thm:main} together with the fact that every isometry of a complete affine building is semisimple, that is, either it fixes a point or it is hyperbolic~\cite[Corollaire~4.2]{Parreau_immeubles}.  (The corresponding consequence of  Theorem~\ref{NOP} is~\cite[Corollary~1.3]{NOP}.)

\begin{corollary}\label{hyperbolic}
If a finitely generated group $G$ acts without a fixed point on a complete 2-dimensional (possibly nondiscrete) affine building $X$ of type $\tA_2$ or $\tC_2$, then $G$ contains a hyperbolic isometry, in particular $\mathbb{Z} \leq G$.
\end{corollary}

\noindent We note that Swenson proved in~\cite[Theorem 11]{Swenson99} that if a group $G$ acts properly discontinuously and cocompactly on an (unbounded, proper) $\CAT(0)$ space $X$, then $G$ has an element of infinite order.  So Corollary~\ref{hyperbolic} can be viewed as a strengthening of this result in some special cases.

Corollary \ref{hyperbolic} in particular gives a negative answer to the question of whether finitely generated infinite torsion groups can act on discrete $2$-dimensional affine buildings (of types $\tA_2$ or $\tC_2$) without fixing a point. This question formed the initial motivation for our work, and was generously suggested to the first author by Pierre-Emmanuel Caprace as a test case to complete a first (small) step towards a Tits Alternative for groups acting on $\CAT(0)$ spaces. 

We prove Theorem~\ref{thm:main} in Section~\ref{sec:proof} by first establishing a result (Corollary~\ref{distance}) concerning the the realisation of distances between fixed point sets of finitely generated subgroups of $G$; this result is immediate for~$X$ discrete.  The proof of Corollary~\ref{distance} given here uses a theorem of Caprace and Lytchak~\cite{CL10} together with properties of panel-trees.  (Recently, Leung, Schillewaert and Thomas~\cite{LeungST} gave an alternative proof of Corollary~\ref{distance}.)  We then reduce to the case that $X$ is a metrically complete $\R$-building and the action of $G$ is type-preserving. We remark that any discrete building is metrically complete, but a nondiscrete building is not necessarily metrically complete, even when it is a Bruhat--Tits building~\cite{MSSS}.  Moreover, the Cauchy completion of a nondiscrete building may not be a building at all~\cite{Kramer}.  In order to apply results for \emph{complete} $\CAT(0)$ spaces as well as for buildings, we use properties of ultrapowers of affine buildings due to Kleiner and Leeb~\cite{KL} and Struyve~\cite{Struyve}.

We next show that if $G$ has two finitely generated proper subgroups whose fixed point sets are nonempty and disjoint, then $G$ contains a hyperbolic element (see Proposition~\ref{prop:hyp}).  Theorem~\ref{thm:main} is obtained by combining this result with an easy induction on the number of generators of~$G$.  To prove Proposition~\ref{prop:hyp}, we construct an orbit of an element $g \in G$ together with a point $\xi \in \partial X$ such that the Busemann function with respect to $\xi$ is unbounded on this orbit, hence~$g$ is hyperbolic.  A key role in this construction is played by the ``local lemmas'' which we establish in Section~\ref{sec:local}.  These  guarantee that for~$\Delta$ a spherical building which occurs as the link of a vertex in~$X$, for any point in $\Delta$ there exists an element $g_\Delta$ of $G$ which acts on $\Delta$ and maps this point ``far away'' from itself.  The element $g$ is then the product of two $g_\Delta$ which are chosen carefully in relation to $\xi$.  The results of Section~\ref{sec:local} for spherical buildings are proved using the description of $\Delta$ as a point-line geometry (which varies by type).  Our arguments in Section~\ref{sec:proof} are partly type-free.

Throughout the paper, we assume knowledge of discrete buildings on the level of the references Abramenko--Brown~\cite{AB} or Ronan~\cite{Ronan}.   Our main reference for nondiscrete affine buildings is Parreau's work~\cite{Parreau_immeubles} (a translation of which into English is available at~\cite{Parreau_translation}).  We use many results from~\cite{Parreau_immeubles}, and mostly follow its terminology and notation.  We also assume basic knowledge of $\CAT(0)$ spaces, see e.g. Bridson--Haefliger~\cite{BH}.

{\bf Acknowledgements}  We thank Martin Bridson, Pierre-Emmanuel Caprace, Timoth\'ee Marquis, Dave Witte Morris, Damian Osajda and James Parkinson for helpful conversations, and an anonymous referee for careful reading, including realising the necessity of Corollary~\ref{distance} in the nondiscrete case.  We also thank the University of Auckland for supporting travel between the first and third authors by means of an  FRDF grant and a PBRF grant.

\section{Local lemmas}\label{sec:local}

In this section we consider the action of a group $G$ on a spherical building $\Delta$ of type $A_2$ or $C_2$.  In Section~\ref{sec:proof}, $\Delta$ will be the link of a vertex of~$X$, hence we refer to the results below as local lemmas. 

We realise $\Delta$ as a $\CAT(1)$ space, so that the distances referred to in each statement are the distances in this metric on $\Delta$.  In particular, opposite points of $\Delta$ are at distance~$\pi$.  For $\Delta$ of type $C_2$ (respectively, $A_2$) we consider $\Delta$ as a generalised quadrangle (respectively, a projective plane).  For $g,h \in G$ and $p$ a panel (point or line) of $\Delta$, we then write $p^g$ for the panel obtained by acting on $p$ by $g$, and put $p^{gh}:=(p^g)^h$. 

Our local lemma in type $C_2$ is as follows.

\begin{lemma}\label{lem:local-C2}  Let $\Delta$ be a building of type $C_2$ (realised as a $\CAT(1)$-space) and let $G$ be a group of type-preserving automorphisms of $\Delta$. If $x$ is a point of $\Delta$ (not necessarily a panel) and $p$ is a panel of $\Delta$ at minimum distance from $x$, then at least one of the following holds:

\begin{compactenum}
\item
There is an element $g \in G$ mapping $p$ to a panel opposite $p$. 
\item
There is a panel $p'$ of $\Delta$ which is fixed by $G$ such that $d(p',x)<\frac{\pi}{2}$.
\end{compactenum}
Moreover in case (1), $d(p,gx) \geq \frac{7\pi}{8}$.
\end{lemma}

\begin{proof}  By duality, we may assume that the panel $p$ is a line $l$.  As $d(p,x)\leq \frac{\pi}{8}$, in case~(1) we obtain that $d(p,gx)\geq \frac{7\pi}{8}$. Now assume we are not in case~(1).  Then $l^g\cap l \neq \emptyset$ for all $g\in G$. If $l^g=l$ for all $g\in G$ then (2) holds with $p'=l$.  So suppose $l^{g_0}\neq l$ for some $g_0\in G$, and let $q=l \cap l^{g_0}$.  If $q\notin l^h$ for some $h \in G \setminus \{ g_0\}$, then $l^h \cap l^{g_0}=\emptyset$ (otherwise there is a triangle). But then $l^{g_0h^{-1}}\cap l=\emptyset$, a contradiction. So $q=\bigcap_{g\in G} l^g$, and hence $q^g=q$ for all $g\in G$.  
Let $p'=q$, then since $q\in l$ and $p=l$ we get $d(p,p')=\pi/4$.  Since $d(x,p)\leq \pi/8$, we conclude $d(x,p')\leq \frac{3\pi}{8}<\pi/2$.
\end{proof}

A different statement is required in type $A_2$, where opposite panels have distinct types.

\begin{lemma}\label{lem:local-A2}
Let $\Delta$ be a building of type $A_2$ (realised as a CAT(1)-space) and $G$ be a group of type-preserving automorphisms of $\Delta$. If $x$ is a point of $\Delta$ (not necessarily a panel), $c$ is a chamber of $\Delta$ containing $x$ and $p$ is a panel of $c$, then at least one of the following holds:

\begin{compactenum}
\item There exists $g \in G$ mapping $c$ to a chamber which contains a panel opposite $p$.
\item There is a panel $p'$ of $\Delta$ which is fixed by $G$ such that $d(p',x)<\frac{\pi}{2}$.
\end{compactenum}
Moreover in case (1), $d(p,gx)\geq \frac{2\pi}{3}$.
\end{lemma}
\begin{proof}
By duality and abuse of notation we may let $c$ be the incident point-line pair $(p,l)$.  Assume first that there is a $g \in G$ such that $p \notin l^g$.  Then $l^g$ is opposite $p$, and (1) holds.  As $G$ is type-preserving and $l^g$ is opposite $p$, $d(p,gx)\geq \frac{2\pi}{3}$, as required. 
Suppose now that $p\in \bigcap_{g\in G} l^g$. If $l^g=l$ for all $g \in G$ then (2) holds, so assume there is a $g\in G$ such that $p = l^g \cap l$. Then $p^h = l^{gh} \cap l^h = p$ for all $h\in G$, and (2) holds.
\end{proof}

\begin{rmk}\label{rem:G2local}
In type $G_2$, opposite panels again have the same type.  However, we cannot expect a local lemma of the same form as Lemma~\ref{lem:local-C2}.  For example, let $\Delta$ be a thin building of type $G_2$.  This has $\CAT(1)$ realisation the subdivision of the unit circle into $12$ arcs, each of length $\frac{\pi}{6}$.  Let $G$ be the group of type-preserving automorphisms of $\Delta$ generated by a rotation~$g$  through angle~$\frac{2\pi}{3}$.  Then for every panel $p$ of $\Delta$, we have $d(p,gp) = d(p,g^2 p) = \frac{2\pi}{3}$.  Since no panel is mapped to an opposite by any element of $G$, and $G$ does not fix any panels, neither (1) nor (2) of Lemma~\ref{lem:local-C2} holds in type~$G_2$.

The best ``local lemma" that we have been able to establish in type $G_2$ is the following trichotomy.  We omit its proof, which is considerably more involved than the proofs of Lemmas~\ref{lem:local-C2} and~\ref{lem:local-A2} above, since we will not actually use this result.  We continue our explanation of why our techniques do not extend to affine buildings of type $\tilde{G}_2$ in Remark~\ref{rem:G2main}.

\begin{lemma}\label{lem:local-G2}
Let $\Delta$ be a building of type $G_2$ (realised as a $\CAT(1)$-space) and let $G$ be a group of type-preserving automorphisms of $\Delta$. If $x$ is a point of $\Delta$ (not necessarily a panel) and $p$ is a panel of $\Delta$ at minimum distance from $x$, then at least one of the following three possibilities must hold:
\begin{compactenum}
\item There is an element $g \in G$ mapping $p$ to a panel opposite $p$. 
\item There are elements $g, h \in G$ such that $p$, $gp$ and $hp$ lie in a common apartment of $\Delta$, and $d(p,gp) = d(gp,hp) = d(hp,p) = \frac{2\pi}{3}$.
\item There is a panel $p'$ of $\Delta$ which is fixed by $G$ such that $d(p',x)<\frac{\pi}{2}$.
\end{compactenum}
Moreover in cases (1) and (2), $d(p,gx) \geq \frac{7\pi}{12}$.
\end{lemma}
\end{rmk}

\section{Proof of the main theorem}\label{sec:proof}

Let $X$ be an affine building, defined as in~\cite[Section 1.2]{Parreau_immeubles}, of type $\tA_2$ or $\tC_2$.  We equip~$X$ with the maximal system of apartments.  Each apartment of $X$ is modelled on the pair $(\Aff,W)$, where~$\Aff$ is a $2$-dimensional real vector space and $W$ is a subgroup of the affine isometry group of $\Aff$ such that the linear part of $W$ is a finite reflection group $\overline{W}$ of type~$A_2$ or $C_2$, respectively.  

Now as in~\cite[Section 1.3.2]{Parreau_immeubles}, each facet of $X$ has a \emph{type} given by the type of the corresponding facet of the fundamental Weyl chamber.  In particular, the codimension one facets of $X$ have just $2$ possible types, with one facet of each of these types bounding each sector (Weyl chamber) in $X$, and the types of facets in any given apartment are invariant under translations of this apartment.  In type  $\tC_2$, each wall of $X$ thus also has a well-defined \emph{type}, induced by the type of the (opposite) facets it contains, and parallel walls in the same apartment have the same type.  We remark that this concept of type is different to the usual definition of type for discrete buildings, where a discrete building of type a rank $n$ Coxeter system has $n$ distinct types of panels.

As in~\cite[Section 6.8]{Rousseau}, we define an \emph{automorphism} of $X$ to be an isometry of $X$ which maps facets to facets and apartments to apartments.  We use this definition rather than the notion of automorphism from~\cite[Definition 2.5]{Parreau_immeubles}, since by~\cite[Proposition 2.5]{Parreau_immeubles} the latter is necessarily type-preserving, and we do not wish to impose this restriction.  

Now let $G$ be a finitely generated group of automorphisms of $X$ such that every element of $G$ fixes a point of~$X$.  
In Section~\ref{sec:distance} we prove a result, Corollary~\ref{distance}, concerning the realisation of distances between fixed point sets of finitely generated subgroups of $G$. We next establish several reductions, in Section~\ref{sec:reductions}. 
Then in Section~\ref{sec:constructions}, assuming $G$ has two proper finitely generated subgroups whose fixed point sets are nonempty and disjoint, we construct an element $g \in G$, a sequence of points $\{ a_i \}$ in $X$ with $a_{2k} = g^k a_1$ for all $k \geq 1$, and a point $\xi$ in $\partial X$, the visual boundary of~$X$.   In Sections~\ref{sec:busemann} and~\ref{sec:end} we show, using these constructions and Busemann functions, that the sequence $\{ a_i \}$ is unbounded, hence $g$ must be hyperbolic, a contradiction.  As explained in Section~\ref{sec:constructions}, combining this with an induction completes the proof of Theorem~\ref{thm:main}.

\subsection{Distance between fixed point sets}\label{sec:distance}

The main result in this section is Corollary~\ref{distance}, which will be essential to our constructions in Section~\ref{sec:constructions}.  

The following lemma and its proof were kindly provided to us by a referee, and shorten our initial argument for Corollary~\ref{distance}.
\begin{lemma} \label{CL-argument}
Let $Y$ be a complete finite-dimensional $\CAT(0)$ space and let $A$ and $B$ be two nonempty closed convex subsets.  Let $d = d(A,B)$.
Then either there exist points $a\in A,b\in B$ such that $d(a,b)=d=d(A,B)$, or there exists $\xi\in \partial A\cap \partial B$.
\end{lemma}
\begin{proof}
For any $R\geq 0$, the $R$-neighbourhood $\mathcal{N}_R(B) =\{y\in Y\mid d(y,B)\leq R\}$
is a closed convex subset of $Y$. Therefore setting $A_n = A \cap \mathcal{N}_{d+\frac{1}{n}}(B)$ for all integers $n>0$, we obtain
a descending chain of closed convex subsets of $A$. If the intersection $\cap_{n} A_n$ is nonempty, then any point $\alpha$ contained in it satisfies $d(\alpha,B) = d =d(A,B)$. Otherwise, \cite[Theorem 1.1]{CL10} ensures the existence
of a point $\xi\in \cap_n \partial A_n$. Now the geodesic ray $[a_1,\xi)$ emanating from any point $a_1\in A_1$ is at bounded distance from $B$ by the definition of $A_1$, and so in particular we have $\xi \in \partial A \cap \partial B$.
\end{proof}

We now set up some further background needed for the proof of Corollary \ref{distance}. The following result is likely known to experts, but we could not find it stated explicitly in the literature.

\begin{lemma}\label{fixedset-apartment}
The fixed point set of a finitely generated group of automorphisms of $X$ intersects each apartment of $X$ in a finite intersection of closed half-apartments.
\end{lemma}
\begin{proof} 
Since the group is finitely generated it suffices to prove the result for an individual automorphism $g$ of $X$.  Let $\App$ be an apartment of $X$.  
Then by \cite[Proposition 9.1]{Rousseau} (see also \cite[Proposition 2.14]{Parreau}), $I:=\App\cap g(\App)$ is a finite intersection of closed half-apartments bounded by walls, 
and moreover $g$ acts like an element of $W$ on $I$.  Now the fixed point set of any element of $W$ acting on an entire apartment is also a finite intersection of closed half-apartments. Therefore the intersection of this fixed point set with $I$ is also such a finite intersection, proving the lemma.
\end{proof}

Given a panel $p_\infty$ of the spherical building at infinity of $X$, the corresponding panel tree $T = T(p_\infty)$ (see, for example, \cite{Tits86}) has as points the equivalence classes of sector panels belonging to the same parallelism class $p_\infty$, where the equivalence relation $\mathcal{R}$ is defined by intersecting in a half-line ($\mathcal{R}$ is transitive since $\CAT(0)$-spaces are uniquely geodesic \cite[Proposition II.1.4]{BH}).  
Let $t_a$ and $t_b$ be two points of $T$ and let $r_a$ and $r_b$ be representatives of $t_a$ and $t_b$, respectively. By \cite[Lemma 4.1]{MSHVM} there exists an apartment containing subrays $r_a'$ of $r_a$ and $r_b'$ of $r_b$. The distance $d(t_a,t_b)$ is defined to be the distance between the parallel halflines $r_a'$ and $r_b'$. This definition is clearly independent of the choices made.

Let $\psi$ be the map which sends a sector panel to its equivalence class with respect to $\mathcal{R}$. We define a map $\phi:X\to T$ as follows. Given a point $x \in X$, let $S_x$ be the (unique) sector panel through $x$ which has $p_\infty$ in its boundary, and define $\phi(x)=\psi(S_x)$.

If a group $G$ acts on $X$ and fixes $p_\infty$ then it has a naturally induced action on $T$. Indeed, for $t\in T$ we define $g(t)$ to be the equivalence class of sector panels which is the image under $g$ of the equivalence class of sector panels corresponding to $t$.

\begin{lemma}\label{fix-on-wall-tree}
Let $G$ be a finitely generated group of elliptic automorphisms of $X$ fixing $p_\infty$.  Then the fixed point set in $T$ of $G$ is the image in $T$ under $\phi$ of the fixed point set in $X$ of $G$.
\end{lemma}
\begin{proof}

If $x\in X$ is fixed by $G$ then since $G$ also fixes $p_\infty$ it fixes the (unique) sector panel through $x$ which has $p_\infty$ in its boundary. Hence by definition $\phi(x)$ is fixed by $G$.
Conversely, if $t \in T$ is fixed by $G$ consider a half-line $L$ belonging to a sector panel in the pre-image $\psi^{-1}(t)$ and $g\in G$.  Then $L^g \cap L$ is a half-line, say $M_g$. Since $g$ is elliptic $M_g$ is pointwise fixed by $g$, as otherwise either $g$ or $g^{-1}$ would map $M_g$ into a proper subset of itself, thus $g$ would act as a translation with axis containing $M_g$ and hence be hyperbolic. Now if $G = \langle g_1,\dots,g_n \rangle$ then the half-line $M = \cap_{i=1}^n M_{g_i}$ is fixed pointwise by $G$. Let $x$ be a point on $M$, then $\phi(x)=t$ and the lemma is proved.
\end{proof}

By \cite[Corollary 2.19]{Parreau}, since we are working with the maximal system of apartments, the $\CAT(0)$ boundary of $X$ is the geometric realisation of its spherical building at infinity.

\begin{corollary}\label{distance} Let $G$ be a finitely generated group of type-preserving automorphisms of 
a complete 2-dimensional Euclidean building $X$.
If $A:= \Fix(G_A)$ and $B:=\Fix(G_B)$ are two nonempty fixed point sets of finitely generated subgroups $G_A$ and $G_B$, then there exist points $\alpha^\star \in A$ and $\beta^\star \in B$ such that $d(\alpha^\star,\beta^\star) = d(A,B)$. 
\end{corollary}
\begin{proof}
By Lemma \ref{CL-argument} we may assume that there exists $\xi\in\partial A\cap \partial B$.   If $\xi$ is not a panel, by Lemma \ref{fixedset-apartment} we may consider sectors $S_a\subseteq A$ and $S_b \subseteq B$ both containing $\xi$ in their boundary and based at $a\in A$ and $b\in B$ respectively. By \cite[Lemma 4.1]{MSHVM} there exists an apartment containing subsectors of $S_a$ and $S_b$ and hence $S_a\cap S_b \neq \emptyset$, thus $A\cap B \neq \emptyset$. Suppose thus that $\xi$ is a panel. By \cite[1.3]{Culler-Morgan} the subsets $T_A$ and $T_B$ of the panel tree $T$ corresponding to $\xi$ which are fixed by $A$ and $B$ respectively are disjoint nonempty closed subtrees of $T$.
Hence by \cite[1.1]{Culler-Morgan} there is a unique shortest geodesic in $T$ having its initial point $t_a$ in $T_A$ and its terminal point $t_b$ in $T_B$.  
Note that by Lemma \ref{fix-on-wall-tree} and \cite[Proposition II.2.2]{BH} no points in $A$ and $B$ can be at distance less than $d(t_a,t_b)$. By definition of distance in the panel tree $d(A,B) = d(t_a,t_b)$ and the proof is complete.\end{proof}

\subsection{Reductions}\label{sec:reductions}

\begin{lemma}\label{special}
We may assume $X$ is an $\mathbb{R}$-building in which each point is a special vertex.
\end{lemma}
\begin{proof}  By the last paragraph in the Remarques on~\cite[p. 6]{Parreau_immeubles} (see also~\cite[Remark 6.3(d)]{Rousseau}), we may regard $X$ as an affine building in which each apartment is modelled on the pair $(\Aff,\widetilde{W})$, where $\widetilde{W}$ is the group of all affine isometries of $\Aff$ whose linear part is $\overline{W}$.  Hence we may assume that $X$ is an $\mathbb{R}$-building in which every point is a special vertex.  We note that if a point $x \in X$ was not originally a special vertex, then the link of $x$ in this $\mathbb{R}$-building structure will be a spherical building in which at least some of the panels are only contained in $2$ chambers.  Thus in general the $\R$-building $X$ will not be thick.
\end{proof}

\begin{lemma}\label{complete} We may pass to the ultrapower of $X$, hence we may in particular assume that $X$ is metrically complete.
\end{lemma}
\begin{proof}
By~\cite[Lemma 4.4]{Struyve} (which uses results from~\cite{KL}) the $\R$-building $X$ can be isometrically embedded in a metrically complete $\R$-building $X'$, the ultrapower of $X$.  Moreover~$X'$ has the same type as $X$, and the $G$-action on $X$ extends to $X'$.  Suppose $G$ fixes a point of $X'$. Then all $G$-orbits on~$X$ are bounded.  Hence as $G$ is finitely generated, by \cite[Main Result 1]{Struyve} $G$ fixes a point of $X$.
\end{proof}

\begin{lemma}\label{type}
Let $G$ be a group of automorphisms of a metrically complete affine building~$X$. 
If its type-preserving subgroup $G'$ fixes a point of $X$, then $G$ fixes a point of~$X$.
\end{lemma}
\begin{proof}
Assume $G'$ fixes $x \in X$. Since $[G:G']$ is finite, the $G$-orbit of $x$ is bounded, hence $G$ fixes a point by the Bruhat--Tits fixed point theorem~\cite[Proposition 3.2.4]{BT}.
\end{proof}

\subsection{Constructions}\label{sec:constructions}

By the results of Section~\ref{sec:reductions}, we may assume from now on that~$X$ is a metrically complete $\R$-building in which each point is a special vertex, and that the action of~$G$ on~$X$ is type-preserving.  

\begin{prop}\label{prop:hyp}  Suppose $G$ has two proper finitely generated subgroups $G_0$ and $G_1$ such that the respective fixed point sets $B_0:=\Fix(G_0)$ and $B_1:=\Fix(G_1)$ are nonempty and disjoint.  Then~$G$ contains a hyperbolic element.
\end{prop}

\begin{proof}[Proof of Theorem~\ref{thm:main}, assuming the result of Proposition~\ref{prop:hyp}]  The group $G$ is finitely generated, with say $G = \langle s_1, \dots, s_n \rangle$.  By hypothesis, each $\langle s_i \rangle$ has a nonempty fixed set.  Then an induction with $G_0 = \langle s_1,\dots,s_i\rangle$ and $G_1 = \langle s_{i+1} \rangle$ completes the proof of Theorem~\ref{thm:main}.
\end{proof}

The remainder of the paper is devoted to the proof of Proposition~\ref{prop:hyp}.  We start with a general result which will be used many times.

\begin{lemma}\label{lem:ray_wall}  Let $\br=[a,\xi)$ and $\mathbf{r'}=[a',\xi)$ be geodesic rays in $X$ which have the same endpoint.  If $\br$ is contained in a wall then so is $\br'$.  Moreover, if $X$ is of type $\tC_2$, then the walls containing $\br$ and $\mathbf{r'}$ are of the same type.
\end{lemma}
\begin{proof}  
Consider a chamber $C$ in the spherical building 
at infinity containing the common endpoint $\xi$ of $\br$ and $\mathbf{r'}$. Let $S$ and $S'$ be the sectors determined by $C$ and respectively $a$ and $a'$ (that is, $S$ is the union of all geodesic rays based at $a$ with endpoint lying in $C$, and similarly for $S'$ and $a'$).  Then $S$ and $S'$ are asymptotic, so by~\cite[Corollary 1.6]{Parreau}, the sectors $S$ and $S'$ have a common subsector $S''$.  Moreover, $S''$ contains a geodesic ray $\mathbf{r''}$ which is parallel to both $\br$ and $\mathbf{r'}$, since $S''$ is contained in both $S$ and $S'$ and has~$\xi$ in its boundary. Since being parallel in an apartment preserves the property of being contained in a wall the first statement of the lemma is proved.  The second statement holds  in type $\tC_2$ because $\br$ and $\mathbf{r'}$ are contained in parallel walls.
\end{proof}

We now define an angle $\alpha$ to equal $\frac{7\pi}{8}$ or $\frac{2\pi}{3}$ as $X$ is of type $\tC_2$ or $\tA_2$, respectively (that is, $\alpha$ is the angle appearing in the corresponding local lemma).    In the remainder of this section, we will construct an element $g \in G$ and a point $\xi$ in $\partial X$, the visual boundary of~$X$, together with a sequence of points $\{ a_i \}_{i=0}^\infty$ of $X$ such that:

\begin{lemma}\label{lem:sequence}
For all $i \geq 1$, we have $\angle_{a_i}(\xi, a_{i+1}) \geq \alpha$.
\end{lemma}

We prove Lemma~\ref{lem:sequence} by induction, with $i=0,1,2$ handled separately first.  The case $i = 1$ includes the construction of $\xi$, and $g$ will be the product of elements of $G$ which appear for $i = 1,2$.  A schematic for our constructions in type $\tC_2$ is given in Figure~\ref{fig:schematic-C2}, where we sketch the geometric situation in a thin building (that is, the Euclidean plane).  The delicate part of the proof is to show, using apartments and retractions, that key portions of this sketch ``lift" to the affine building~$X$.  We give further figures in both types below.

\medskip
\begin{figure}[ht]
\begin{overpic}[width=0.9\textwidth]{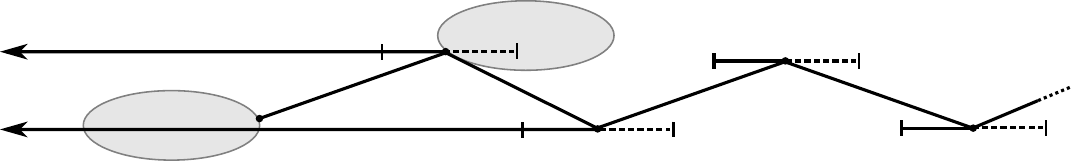} 
\put(-2,2){$\xi$}
\put(-2,9.5){$\xi$}
\put(14,4){$B_0$}
\put(52,11){$B_1$}
\put(21.5,3.5){\footnotesize{$a_0$}}
\put(41.5,11){\footnotesize{$a_1$}}
\put(35,11.5){\footnotesize{$p_1$}}
\put(46.5,11.5){\footnotesize{$g_1p_1$}}
\put(55,1.5){\footnotesize{$a_2$}}
\put(47.5,1){\footnotesize{$p_2$}}
\put(61,1){\footnotesize{$g_2p_2$}}
\put(72.5,10){\footnotesize{$a_3$}}
\put(66,10.5){\footnotesize{$p_3$}}
\put(78.5,10.5){\footnotesize{$g_3p_3$}}
\put(90,1.5){\footnotesize{$a_4$}}
\put(83.5,1){\footnotesize{$p_4$}}
\put(96,1){\footnotesize{$g_4p_4$}}
\put(40.5,6){\footnotesize{$g_1$}}
\put(55,6){\footnotesize{$g_2$}}
\put(72,6){\footnotesize{$g_3$}}
\put(90,6){\footnotesize{$g_4$}}
\end{overpic}
\caption{A schematic for our constructions in type $\tC_2$.  Here, we will define $g = g_2 g_1$, and $a_i = g a_{i-2}$ and $g_i = g g_{i-2} g^{-1}$ for all $i \geq 3$. }
\label{fig:schematic-C2}
\end{figure}

Now, as in the statement of Proposition~\ref{prop:hyp}, assume that $G$ has two proper finitely generated subgroups $G_0$ and $G_1$ such that the respective fixed point sets $B_0:=\Fix(G_0)$ and $B_1:=\Fix(G_1)$ are nonempty and disjoint.  Since $X$ is a complete $\CAT(0)$ space, for $i = 0,1$ we have that $B_i$ is a convex subset of~$X$ \cite[Corollary~II.2.8(1)]{BH}.  Note also that $B_i$ is closed, hence complete in the induced metric, since $G$ acts isometrically.
Hence Corollary \ref{distance} guarantees the existence of points $a_0 \in B_0$ and $a_1 \in B_1$ such that $d(a_0,a_1)=d(B_0,B_1)>0$.  Note that then $a_0$ (respectively, $a_1$) is the closest-point projection of $a_1$ (respectively, $a_0$) to $B_0$ (respectively, $B_1$).

From now on, for $i \geq 1$ we write $\Delta_i$ for the spherical building which is the link of $a_i$ in~$X$, and $x_i$ for the point of $\Delta_i$ corresponding to the geodesic segment $[a_{i-1},a_i]$ (it will be seen from the construction that $a_{i-1}$ and $a_i$ are always distinct).  By our assumption that every point of $X$ is a special vertex, each $\Delta_i$ will have the same type.  Note that the Alexandrov angle at $a_i$ between any two points $x$ and $y$ of $\Delta_i$, denoted $\angle_{a_i}(x,y)$, is equal to the distance between $x$ and~$y$ in the $\CAT(1)$ metric on~$\Delta_i$. 

If $\Delta_1$ is of type $C_2$, let $p_1$ be a panel of $\Delta_1$ at minimum distance from $x_1$.  If~$\Delta_1$ is of type $A_2$, let $c_1$ be a chamber of $\Delta_1$ containing $x_1$ and let $p_1$ be a panel of $c_1$.  The next, crucial, result uses the local lemmas from Section~\ref{sec:local}.

\begin{lemma}\label{lem:a1}  
\begin{compactenum}
\item If $\Delta_1$ is of type $C_2$, then there is a $g_1 \in G_1$ mapping $p_1$ to a panel of~$\Delta_1$ which is opposite $p_1$.  
\item If $\Delta_1$ is of type $A_2$, then there is a $g_1 \in G_1$ mapping $c_1$ to a chamber of $\Delta_1$ which contains a panel $p_1^\op$ opposite $p_1$.  
\end{compactenum}
Moreover, in both cases, $\angle_{a_1}(p_1,g_1x_1) \geq \alpha$.
\end{lemma}
\begin{proof}  Since $G_1$ fixes $a_1$, it acts on $\Delta_1$.   It suffices to show that case~(2) in Lemma~\ref{lem:local-C2} or Lemma~\ref{lem:local-A2} 
does not occur (with the obvious modifications of notation).  Assume by contradiction that there is a panel $p'$ of $\Delta_1$ which is fixed by $G_1$ and 
is such that $d(p',x_1)<\pi/2$ in the $\CAT(1)$ metric on $\Delta_1$. 

Since $p'$ is fixed, every generator of $G_1$ will fix a line segment $[a_1,x_i]$ with $x_i \neq a_1$, and since $G_1$ is finitely generated there exists a $j$ such that $[a_1,x_j]\subset B_1$. Hence there exists an $x':=x_j \neq a_1 \in B_1$ and we have $\angle_{a_1}(a_0,x')<\pi/2$. This contradicts ~\cite[Proposition II.2.4(3)]{BH}, completing the proof.
\end{proof}

Let $g_1 \in G_1$ be as given by Lemma~\ref{lem:a1}, and define $a_2 := g_1 a_0$. 

It will be helpful from here on to abuse terminology, as follows.  If $p$ is a panel of $\Delta_i$ and $M$ is the wall of an apartment of $X$ containing~$a_i$ which is determined by $p$, then we will say that $M$ contains $p$.  We similarly abuse terminology for geodesic rays based at~$a_i$ which are contained in walls determined by $p$.

We now construct $\xi$, and prove Lemma~\ref{lem:sequence} for $i = 1$.  
Let $S_{12}$ be a sector of $X$ based at~$a_1$ which contains the point $a_2$, so that in type $\tilde A_2$, if $a_2$ lies on a wall through $a_1$ then $S_{12}$ is a sector corresponding to the chamber $g_1 c_1$ of $\Delta_1$.

\begin{lemma}\label{lem:A1}  There is an apartment $\App_1$ containing $S_{12}$, such that $p_1$ is contained in a wall $M_1$ of $\App_1$.
\end{lemma}

\begin{proof}  By Lemma~\ref{lem:a1} we have that $\angle_{a_1}(p_1,g_1x_1) \geq \alpha$, hence the sectors $S_{12}$ and $g_1^{-1}(S_{12})$ are distinct.  So by~\cite[Proposition 1.15]{Parreau}, there is an apartment $\App_1$ containing both $S_{12}$ and a germ of $g_1^{-1}(S_{12})$.  Now define $M_1$ to be the wall of $\App_1$ through $a_1$ which contains $p_1$.
\end{proof}  
  
We define $\xi \in \partial X$ to be the endpoint of $M_1$ such that the ray $[a_1, \xi)$ contains $p_1$.  By construction $\angle_{a_1}(\xi,a_2) = \angle_{a_1}(p_1,g_1x_1)$.  
Hence by Lemma~\ref{lem:a1} we have $\angle_{a_1}(\xi,a_2)  \geq \alpha$.  This proves Lemma~\ref{lem:sequence} for $i = 1$.

For $i \geq 1$, we now define $\br_i$ to be the ray $[a_i,\xi)$.  By construction, $\br_1 = [a_1,\xi)$ is contained in a wall, and so by Lemma~\ref{lem:ray_wall} and induction, for all $i \geq 1$ the ray $\br_i$ will be contained in a wall.  We may thus, for $i \geq 1$, define $p_i$ to be the panel of $\Delta_i$ induced by $\br_i$.  By Lemma~\ref{lem:ray_wall} again, if $X$ is of type $\tC_2$, the panels $p_i$ are all of the same type.

We next prove Lemma~\ref{lem:sequence} for $i = 2$.  Define $M_2'$ to be the wall of $\App_{1}$ containing $\br_2$.

\begin{lemma}\label{lem:p2} 
\begin{compactenum}
\item If $\Delta_2$ is of type $C_2$, then $p_2$ is a panel of $\Delta_2$ which is at minimum distance from $x_2$.
\item If $\Delta_2$ is of type $A_2$, then there is a chamber $c_2$ of $\Delta_2$ which contains $x_2$ such that $p_2$ is a panel of $c_2$.  
\end{compactenum}
\end{lemma}
\begin{proof}  
If $\Delta_2$ is of type $C_2$ (see Figure~\ref{fig:Lemma312typeC2}), then by construction the wall $M_{1}$ through $a_{1}$ contains both $p_{1}$ and $g_{1} p_{1}$.  Now $\angle_{a_{1}}(x_{1},p_{1}) = \angle_{a_{1}}(g_{1} x_{1}, g_{1} p_{1})$, both $g_{1}x_{1}$ and $x_2$ lie on the geodesic segment $[a_{1},a_2]$, and $M_1$ and $M_2'$ are parallel.  It follows that $\angle_{a_2}(x_2, p_2) = \angle_{a_1}(x_1,p_1)$, and thus $p_2$ is a panel of $\Delta_2$ at minimum distance from $x_2$. 

\bigskip
\begin{figure}[ht]
\begin{overpic}[width=0.6\textwidth]{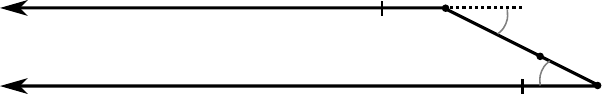} 
\put(-3,0){$\xi$}
\put(-3,13){$\xi$}
\put(10,-3){$M_2'$}
\put(10,16){$M_1$}
\put(30,-2){$\br_2$}
\put(30,16){$\br_{1}$}
\put(85,-3){$p_2$}
\put(62,17){$p_1$}
\put(98,-2){$a_2$}
\put(72,16){$a_1$}
\put(90,7){$x_2$}
\put(80,16){$g_1p_1$}
\end{overpic}
\caption{The proof of Lemma~\ref{lem:p2} in type $C_2$.}
\label{fig:Lemma312typeC2}
\end{figure}

\begin{figure}[ht]
\begin{overpic}[width=0.6\textwidth]{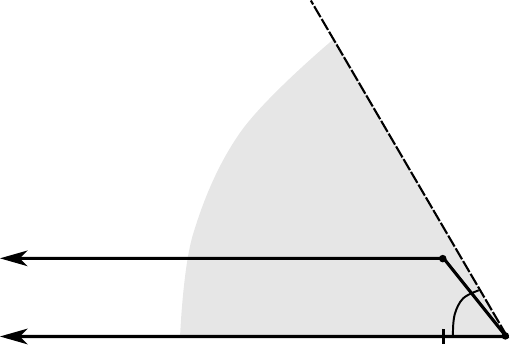} 
\put(-3,0){$\xi$}
\put(-3,16){$\xi$}
\put(10,-3){$M_2'$}
\put(10,19){$M_1$}
\put(30,-2){$\br_2$}
\put(30,19){$\br_{1}$}
\put(85,-3){$p_2$}
\put(98,-2){$a_2$}
\put(84,19){$a_1$}
\put(91,5){$c_2$}
\put(55,25){$S_2$}
\end{overpic}
\caption{The proof of Lemma~\ref{lem:p2} in type $A_2$.}
\label{fig:Lemma312typeA2}
\end{figure}

If $\Delta_2$ is of type $A_2$ (see Figure~\ref{fig:Lemma312typeA2}), there is a sector $S_2$ of $\App_1$ which is based at $a_2$ and is bounded by $\br_2$, such that $a_1$ is in $S_2$.  Then we take $c_2$ to be the chamber of $\Delta_2$ determined by $S_2$. 
\end{proof}

Now define $G_2 := g_1 G_0 g_1^{-1}$ and $B_2 := g_1 B_0$, so that $B_2$ is the fixed set of $G_2$.  Then since $g_1$ is an isometry which fixes $B_1$, we have $d(B_0,B_1) = d(B_1,B_2) = d(a_1,a_2)$, and that~$a_2$ is the closest-point projection of $a_1$ to~$B_2$.  
The next result is proved using the local lemmas of Section~\ref{sec:local} and similar arguments to those in the proof of Lemma~\ref{lem:a1}.

\begin{lemma}\label{lem:a2}  
\begin{compactenum}
\item If $\Delta_2$ is of type $C_2$, there is a $g_2 \in G_2$ mapping $p_2$ to a panel of $\Delta_2$ which is opposite $p_2$.  
\item If $\Delta_2$ is of type $A_2$, let $c_2$ be as given by Lemma~\ref{lem:p2}(2).  Then there is a $g_2 \in G_2$ mapping $c_2$ to a chamber of $\Delta_2$ which contains a panel $p_2^\op$ opposite $p_2$.  
\end{compactenum}
Moreover, in all cases, $\angle_{a_2}(p_2,g_2x_2) \geq \alpha$.
\end{lemma}

We now define
\[ g: = g_2 g_1 \in G. \]
Observe that since $g_1$ fixes $a_1$ and $g_2$ fixes $a_2 = g_1 a_0$, we have $a_2 = g a_0$ and $a_3 = g a_1$.  We may thus, for $i \geq 2$, inductively define
\[
a_i := g a_{i-2}.
\]
We also for $i \geq 3$ define \[ g_i := g g_{i-2} g^{-1}.\]  An easy induction shows that $g_{i-1}g_{i-2}=g$ for all $i \geq 3$.
Also, for all $i \geq 1$, by induction $g_i$ fixes $a_i$, hence $g_i$ acts on $\Delta_i$, and we have $g_i a_{i-1} = a_{i+1}$.  Finally, if $\Delta_i$ is of type $A_2$ then for $i \geq 3$ we define (with abuse of terminology using chambers in the spherical buildings instead of their corresponding sectors in $X$)
\[
c_i := g c_{i-2}.
\]

The next result completes the proof of Lemma~\ref{lem:sequence}. 

\begin{lemma}\label{lem:induct}  For $i \geq 1$: 
\begin{enumerate}
\item If $\Delta_i$ is of type $C_2$, then $p_i$ is equal to $g p_{i-2}$ and is a panel of $\Delta_i$ which is at minimum distance from~$x_i$.  Moreover, $g_i p_i$ is opposite $p_i$. 
\item If $\Delta_i$ is of type $A_2$, then $c_i$ contains $x_i$, $p_i$ is a panel of $c_i$ and $g_i c_i$ contains a panel $p_i^\op$ which is opposite $p_i$.  
\end{enumerate}
Moreover, in all cases, $\angle_{a_i}(\xi, a_{i+1}) \geq \alpha$. 
\end{lemma}

\begin{proof}  The proof is by induction on $i$, and the cases $i=1,2$ have been established above.  

\bigskip 

\begin{figure}[ht]
\begin{overpic}[width=0.6\textwidth]{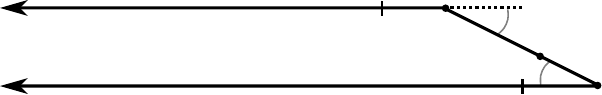} 
\put(-3,0){$\xi$}
\put(-3,13){$\xi$}
\put(10,-3){$M_i'$}
\put(10,16){$M_{i-1}$}
\put(30,-2){$\br_i$}
\put(30,16){$\br_{i-1}$}
\put(85,-3){$p_i$}
\put(62,18){$p_{i-1}$}
\put(98,-2){$a_i$}
\put(72,16){$a_{i-1}$}
\put(90,7){$x_i$}
\end{overpic}
\caption{The proof of Lemma~\ref{lem:induct} in type $C_2$.}
\label{fig:Lemma315typeC2}
\end{figure}

Suppose first that $\Delta_i$ is of type $C_2$.  For $i\geq 3$, by induction the panel $gp_{i-2}$ is at minimum distance from $gx_{i-2} = x_i$.  Since the action of $G$ is type-preserving, $g p_{i-2}$ has the same type as $p_{i-2}$.  Thus $p_i$ and $g p_{i-2}$ have the same type.  

We next show that $p_i$ is a panel at minimum distance from $x_i$ (see Figure~\ref{fig:Lemma315typeC2}).  Let $S_{i-1}$ be a sector of $X$ based at $a_{i-1}$ which contains the point $a_i$.  By induction, the panels $p_{i-1}$ and $g_{i-1}p_{i-1}$ are opposite, and so the union of $\br_{i-1}$ with the facet of $S_{i-1}$ containing $g_{i-1}p_{i-1}$ is a wall $M_{i-1}$ of $X$.  Moreover the wall $M_{i-1}$ bounds both $S_{i-1}$ and a sector $S_{i-1}^\op$ of $X$ which is opposite $S_{i-1}$.  Thus by~\cite[Proposition~1.12]{Parreau}, there exists a (unique) apartment $\App_{i-1}$ containing both $S_{i-1}^\op$, hence $\br_{i-1}$, and $S_{i-1}$.  Let $M_{i}'$ be the wall of $\App_{i-1}$ containing $\br_i$.
Now observe that as $M_{i-1}$ and $M_i'$ are parallel, we have $\angle_{a_{i-1}}(g_{i-1} p_{i-1}, g_{i-1} x_{i-1}) = \angle_{a_i}(p_i,x_i)$.   Also, by induction, $g_{i-1} p_{i-1}$ is a panel at minimum distance from $g_{i-1} x_{i-1}$.  Therefore $p_i$ is a panel at minimal distance from $x_i$.  

Now $p_i$ and $g p_{i-2}$ are panels of the same type both at minimum distance from $x_i$.  It follows that $p_i = g p_{i-2}$.  
Hence as $p_{i-2}$ and $g_{i-2} p_{i-2}$ are opposite, the panels $p_i = g p_{i-2}$ and $g g_{i-2} p_{i-2}$ are opposite.  To complete the proof of (1) in type $C_2$, we observe that
\[
g_i p_i = g g_{i-2}g^{-1} g p_{i-2} = gg_{i-2}p_{i-2}.
\]

If $\Delta_i$ is of type $A_2$, then since $c_i = g c_{i-2}$ and $x_i = g x_{i-2}$, by induction $c_i$ contains $x_i$.  We next show that $p_i = gp_{i-2}$, which implies that $p_i$ is a panel of $c_i$.  
Now $c_{i-1}$ is a chamber of $\Delta_{i-1}$ which contains $x_{i-1}$ and $p_{i-1}$ is a panel of $c_{i-1}$.  Hence as $G$ is type-preserving and opposite panels in type $A_2$ have distinct types, $g_{i-1} c_{i-1}$ has panels $p_{i-1}^\op$ and $g_{i-1} p_{i-1}$ (and contains $g_{i-1} x_{i-1}$).  Let $S_{i-1}$ be a sector of $X$ which is based at $a_{i-1}$ and contains $a_{i-2}$.  We consider two cases.

\bigskip

\begin{figure}[ht]
\begin{overpic}[width=0.6\textwidth]{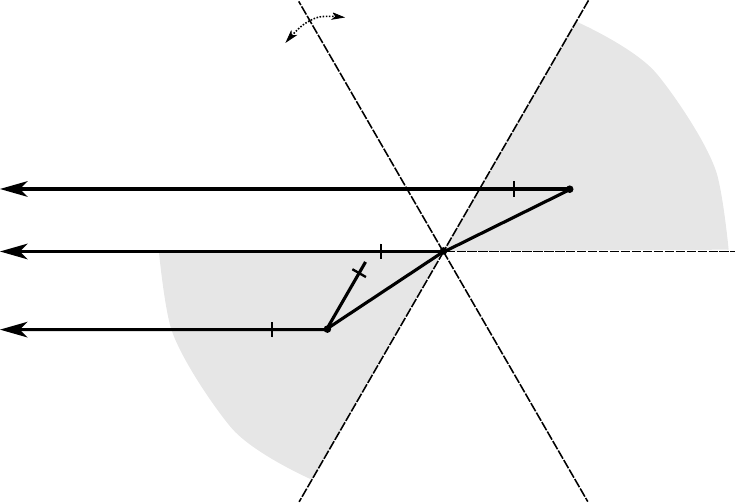} 
\put(-3,44){$\xi$}
\put(-3,33){$\xi$}
\put(-3,22){$\xi$}
\put(8,44){\footnotesize{$M_i''$}}
\put(8,36){\footnotesize{$M_{i-1}''$}}
\put(8,25){\footnotesize{$M_{i-2}''$}}
\put(69,45){\footnotesize{$p_i$}}
\put(51,36){\footnotesize{$p_{i-1}$}}
\put(35,26){\footnotesize{$p_{i-2}$}}
\put(34,31){\footnotesize{$g_{i-2}p_{i-2}$}}
\put(79,43){\footnotesize{$a_i$}}
\put(63,32){\footnotesize{$a_{i-1}$}}
\put(46,23){\footnotesize{$a_{i-2}$}}
\put(35,12){\footnotesize{$S_{i-1}$}}
\put(76,54){\footnotesize{$g_{i-1}S_{i-1}$}}
\put(43,68){\footnotesize{$r_{i-1}$}}
\end{overpic}
\caption{The proof of Lemma~\ref{lem:induct} in type $A_2$, Case I.}
\label{fig:Lemma315typeA2opposite}
\end{figure}

\emph{Case I: $g_{i-1} c_{i-1}$ and $c_{i-1}$ are opposite.}    See Figure~\ref{fig:Lemma315typeA2opposite}.  In this case, since $g_{i-1} S_{i-1}$ and $S_{i-1}$ are opposite, by~\cite[Proposition~1.12]{Parreau} there is a unique apartment $\App_{i-1}$ which contains both $S_{i-1}$ and $g_{i-1} S_{i-1}$.  Let $r_{i-1}$ be the reflection of $\App_{i-1}$ in its unique wall which passes through $a_{i-1}$ and does not bound $S_{i-1}$ (or $g_{i-1} S_{i-1}$).  That is, $r_{i-1}$ is the reflection of $\App_{i-1}$ which fixes $a_{i-1}$ and takes $p_{i-1}$ to $g_{i-1} p_{i-1}$.  Then the geodesic segment $[a_{i-1},a_i] = g_{i-1}[a_{i-1},a_{i-2}]$ of $\App_{i-1}$ is obtained from $[a_{i-1},a_{i-2}]$ by applying the reflection $r_{i-1}$.  For $j = i-2,i-1,i$, write $M_j''$ for the wall of $\App_{i-1}$ which passes through $a_j$ and contains $p_j$.  Then since each $M_j''$ contains at least some initial portion of the geodesic ray $\br_j$, the three walls $M_{i-2}''$, $M_{i-1}''$ and $M_i''$ of $\App_{i-1}'$ are mutually parallel.  It follows that $r_{i-1}$ maps the wall of $\App_{i-1}$ which passes through $a_{i-2}$ and contains $g_{i-2} p_{i-2}$ to $M_i''$.  Hence $g_{i-1} g_{i-2} p_{i-2} = g p_{i-2} = p_i$ in this case.

\bigskip

\begin{figure}[ht]
\begin{overpic}[width=0.6\textwidth]{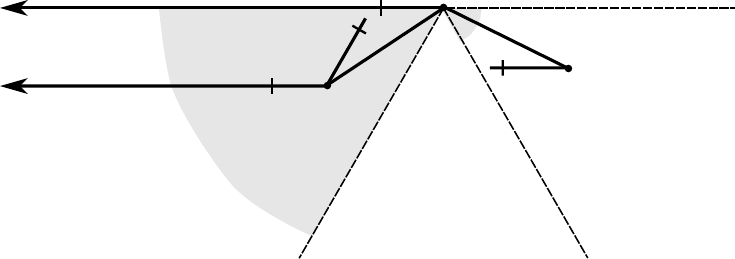} 
\put(-3,33){$\xi$}
\put(-3,22){$\xi$}
\put(51,36){\footnotesize{$p_{i-1}$}}
\put(35,26){\footnotesize{$p_{i-2}$}}
\put(34,31){\footnotesize{$g_{i-2}p_{i-2}$}}
\put(78,26){\footnotesize{$\rho_{i-1}(a_{i})$}}
\put(64,22.5){\footnotesize{$\rho_{i-1}(p_{i})$}}
\put(60,35){\footnotesize{$a_{i-1}$}}
\put(46,23){\footnotesize{$a_{i-2}$}}
\put(35,12){\footnotesize{$S_{i-1}$}}
\put(78,15){\footnotesize{$\rho_{i-1}(g_{i-1}S_{i-1})$}}
\end{overpic}
\caption{The proof of Lemma~\ref{lem:induct} in type $A_2$, Case II.}
\label{fig:Lemma315typeA2nonopposite}
\end{figure}

\emph{Case II: $g_{i-1} c_{i-1}$ and $c_{i-1}$ are not opposite.}  See Figure~\ref{fig:Lemma315typeA2nonopposite}.  Note that these chambers cannot be adjacent in $\Delta_{i-1}$, as $g_{i-1} c_{i-1}$ contains a panel opposite to the panel $p_{i-1}$ of $c_{i-1}$.  By~\cite[Proposition 1.15]{Parreau}, there is an apartment $\App_{i-1}$ which contains $S_{i-1}$ and a germ of $g_{i-1} S_{i-1}$.  Write $\rho_{i-1}$ for the retraction of $X$ onto $\App_{i-1}$ such that $\rho_{i-1}^{-1}(a_{i-1}) = \{ a_{i-1} \}$, as guaranteed by~\cite[Axiom (A5$'$)]{Parreau}.  Then by~\cite[Proposition 1.17]{Parreau}, $\rho_{i-1}$ maps $g_{i-1}S_{i-1}$ isometrically onto the sector $\rho_{i-1}(g_{i-1}S_{i-1})$ of $\App_{i-1}$ which is based at $a_{i-1}$ and has the same germ as $g_{i-1}S_{i-1}$.  Thus $d(a_{i-1}, a_{i-2}) = d(a_{i-1}, a_i) = d(a_{i-1}, \rho_{i-1}(a_i))$, and $S_{i-1}$ and $\rho_{i-1}(g_{i-1} S_{i-1})$ are nonadjacent and nonopposite sectors of $\App_{i-1}$ which respectively contain the geodesic segments $[a_{i-1},a_{i-2}]$ and $[a_{i-1}, \rho_{i-1}(a_i)]$.  Hence $[a_{i-1},\rho_{i-1}(a_i)]$ is obtained from $[a_{i-1},a_{i-2}]$ by applying a rotation of $\App_{i-1}$ about the point $a_{i-1}$ through angle $2\pi/3$.  This rotation takes the wall of $\App_{i-1}$ through $a_{i-2}$ which contains $g_{i-2} p_{i-2}$ to the wall of $\App_{i-1}$ through $\rho_{i-1}(a_i)$ which contains $\rho_{i-1}(p_i)$.  Since retractions and this rotation are type-preserving, it follows that $g_{i-1} g_{i-2} p_{i-2} = g p_{i-2} = p_i$, as required.

We have now shown that $p_i = g p_{i-2}$.  Since $g_{i-2} c_{i-2}$ contains $p_{i-2}^\op$, which is opposite $p_{i-2}$, we have that 
 $gg_{i-2}c_{i-2}$ contains $g(p_{i-2}^\op)$, which is opposite $p_i = gp_{i-2}$.  Also, 
\[
g_i c_i = gg_{i-2}g^{-1} gc_{i-2} = gg_{i-2}c_{i-2}.
\]
So $g_ic_i$ contains a panel $p_i^{\mathrm{op}}$ which is opposite $p_i$, as required to finish the proof of (2).

Now the same arguments as in the cases $i = 1,2$ show that $\angle_{a_i}(p_i, g_i x_i) \geq \alpha$.  
By construction we have $\angle_{a_i}(p_i,g_i x_i) = \angle_{a_i}(\xi, a_{i+1})$, which completes the proof.
\end{proof}

\begin{rmk}\label{rem:G2main}  Despite substantial effort we have been unable to extend our approach to type~$\tG_2$. First, as explained in Remark~\ref{rem:G2local}, there cannot be a local lemma for $G_2$ which  guarantees that we obtain a panel opposite to $p_i$ at every step.  Hence we cannot make use of opposite sectors in $X$ as we did in type $\tC_2$ above. 

Moreover, the local lemma that we have been able to prove in type $G_2$ (see Lemma~\ref{lem:local-G2}) does not give us enough control to be able to show that $p_i = g p_{i-2}$, as we did in type $\tA_2$.  The real issue is that for $X$ of type $\tG_2$, if possibility (2) in Lemma~\ref{lem:local-G2} occurs at step $i$, then the rays $\br_{i}$ and $\br_{i+1}$ have germs which may or may not be parallel in an apartment of $X$ which contains both of these germs.  This makes it very difficult to run an inductive procedure.  We have also tried defining $g = g_4 g_3 g_2 g_1$ in various ways, rather than using $g = g_2 g_1$ as above, but this just postpones the problem. 
\end{rmk}

\subsection{Unbounded Busemann function} \label{sec:busemann}

For any metric space $Y$, the Busemann function (see \cite[Definition II.8.17]{BH}) associated to a geodesic ray $\gamma$ in $Y$ is given by, for $y \in Y$, \[ b_\gamma(y):= \lim_{t\to \infty} [d(y,\gamma(t))-t]. \] 
We will apply the following general result.

\begin{lemma}\label{lem:busemann}  Let $Y$ be a complete $\CAT(0)$ space.  Let $y \in Y$, let $\eta \in \partial Y$ and let $\gamma$ be the geodesic ray $[y,\eta)$.  Suppose $z \in Y$ is such that $d(y,z) = D > 0$ and $\angle_{y}(\eta,z) = \theta > \frac{\pi}{2}$.  Then
\[
b_\gamma(z) \geq D \sin \left( \theta - \frac{\pi}{2} \right) > 0.
\]
\end{lemma}
\begin{proof}  Fix $t > 0$ and write $y_t = \gamma(t)$, so that $\theta = \angle_{y}(\eta,z) = \angle_y(y_t, z)$.  Consider a triangle in the Euclidean plane $\mathbb{E}^2$ with vertices $\hat{y_t}$, $\hat{y}$ and $\hat{z}$, so that $d(\hat{y},\hat{y_t}) = d(y, y_t) = t$, $d(\hat{y},\hat{z}) = d(y,z) = D$ and $\angle_{\hat{y}}(\hat{y_t},\hat{z}) = \theta$.  Then by~\cite[Proposition II.1.7(5)]{BH}, we have $d(z,y_t) \geq d(\hat{z},\hat{y_t})$.  

Let $\ell$ be the line in $\mathbb{E}^2$ which extends $[\hat{y_t},\hat{y}]$ and let $p_\ell: \mathbb{E}^2 \to \ell$ be the closest-point projection.  Then since $\theta > \frac{\pi}{2}$, the point $\hat{y}$ lies strictly between the points $\hat{y_t}$ and $p_\ell(\hat{z})$ on~$\ell$.  If $\theta = \pi$, equivalently $\hat{z}$ lies on $\ell$, then \[ d(\hat{z},\hat{y_t}) = d(\hat{z},\hat{y}) + d(\hat{y},\hat{y_t}) = D + t = D\sin\left(\theta - \frac{\pi}{2}\right) + t.\]  Otherwise, by considering the right-angled Euclidean triangle with vertices $\hat{y}$, $p_\ell(\hat{z})$ and~$\hat{z}$, we calculate $d(\hat{y}, p_\ell(\hat{z})) = D \sin (\theta - \frac{\pi}{2})$.  Hence as $[\hat{z}, \hat{y_t}]$ is the hypotenuse of the right-angled Euclidean triangle with vertices $\hat{y_t}$, $p_\ell(\hat{z})$ and $\hat{z}$, we obtain \[
d(\hat{z}, \hat{y_t}) > d(p_\ell(\hat{z}), \hat{y_t}) = d(p_\ell(\hat{z}),\hat{y}) + d(\hat{y},\hat{y_t}) = D \sin \left(\theta - \frac{\pi}{2}\right) + t.\]

We have now shown that for all $t > 0$,
\[
d(z, \gamma(t)) - t = d(z, y_t) - t \geq D \sin \left( \theta - \frac{\pi}{2}\right).
\]
Taking the limit as $t \to \infty$ gives the desired result.
\end{proof}

Now let $X$ be as in Section~\ref{sec:constructions}, and let $\{ a_i \}_{i=0}^\infty \subset X$ and $\xi \in \partial X$ be as constructed there.  Recall that $\br_i$ is the geodesic ray $[a_i,\xi)$.

\begin{corollary}\label{unbounded}  
$\lim_{i\to \infty} b_\mathbf{r_1}(a_i) = +\infty$.
\end{corollary}
\begin{proof}
Let $D=d(a_0,a_1) = d(a_i,a_{i+1}) > 0$.  By Lemma~\ref{lem:sequence}, for all $i \geq 1$ we have $\angle_{a_i}(\xi,a_{i+1}) \geq \alpha > \frac{\pi}{2}$.  Thus by Lemma~\ref{lem:busemann}, for all $i \geq 1$
\[
b_{\br_i}(a_{i+1}) \geq D \sin \left( \alpha - \frac{\pi}{2} \right) > 0.
\]
As $b_{\br_i}(a_i) = 0$, in fact
\[
b_{\br_i}(a_{i+1}) - b_{\br_i}(a_i) \geq D \sin \left(\alpha - \frac{\pi}{2} \right) > 0.
\]
Now for all $i \geq 1$ the rays $\br_i$ are asymptotic, since they all have endpoint~$\xi$.  Hence the Busemann functions $b_{\br_i}$ pairwise differ by a constant~\cite[Corollary II.8.20]{BH}.   Thus the difference $b_{\br_1}(a_{i+1}) -b_{\br_1}(a_{i})$ has a strictly positive lower bound which is independent of~$i$, proving the result. 
\end{proof}

\subsection{End of proof of Proposition~\ref{prop:hyp}}\label{sec:end}
If $g$ were elliptic then $a_0$ would have a bounded orbit under $g$.  By definition, $a_{2k}$ equals $g^{k}a_0$ for all $k \geq 0$, and by Corollary~\ref{unbounded} we have $\lim_{k \to \infty} b_{\br_1}(a_{2k}) = + \infty$.  But $b_{\br_1}(a_{2k}) - b_{\br_1}(a_0)$ is at most $d(a_0,a_{2k})$, a contradiction.  Thus by~\cite[Corollaire~4.2]{Parreau_immeubles}, $g$ is hyperbolic.

\renewcommand{\refname}{Bibliography}
\bibliography{bibliography}
\bibliographystyle{siam}

\end{document}